\documentclass[12pt,twoside]{amsart}
\usepackage{amssymb,amsmath,amsthm, amscd, enumerate, mathrsfs}
\usepackage{graphicx, hhline}
\usepackage[all]{xy}
\title{Quasi-log canonical pairs are Du Bois}
\author{Osamu Fujino and Haidong Liu}
\date{2019/6/8, version 0.10}
\address{Department of Mathematics, Graduate School of Science, 
Osaka University, Toyonaka, Osaka 560-0043, Japan}
\email{fujino@math.sci.osaka-u.ac.jp}
\address{Department of Mathematics, Graduate School of Science, 
Kyoto University, Kyoto 606-8502, Japan}
\email{jiuguiaqi@gmail.com}
\DeclareMathOperator{\Nqklt}{Nqklt}
\DeclareMathOperator{\Gr}{Gr}
\DeclareMathOperator{\Supp}{Supp\!}

\newtheorem{thm}{Theorem}[section]
\newtheorem{lem}[thm]{Lemma}

\newtheorem{cor}[thm]{Corollary}

\theoremstyle{definition}
\newtheorem{defn}[thm]{Definition}

\newtheorem*{ack}{Acknowledgments} 
\newtheorem*{conventions}{Conventions}         

\makeatletter
    
    \@addtoreset{equation}{section}
\makeatother

\begin{document}

\begin{abstract}
We prove that every quasi-log canonical pair has 
only Du Bois singularities. Note that our arguments are free from the 
minimal model program. 
\end{abstract}

\maketitle 

\tableofcontents

\section{Introduction}\label{f-sec1} 
The notion of quasi-log canonical pairs was first introduced by 
Florin Ambro (see \cite{ambro}) 
and is now known to be ubiquitous in the theory of 
minimal models. 
The main purpose of this paper is to establish: 

\begin{thm}\label{f-thm1.1}
Every quasi-log canonical pair has only Du Bois singularities. 
\end{thm}

Theorem \ref{f-thm1.1} is a complete generalization 
of \cite[Theorem 1.4]{kollar-kovacs} and \cite[Corollary 6.32]{kollar}. 
This is because we get the following corollary by combining 
Theorem \ref{f-thm1.1} with the main result of \cite{fujino-slc}. 

\begin{cor}[{\cite[Corollary 6.32]{kollar}}]\label{f-cor1.2}
Let $(X, \Delta)$ be a semi-log canonical 
pair. 
Then any union of slc strata of $(X, \Delta)$ is Du Bois. 
In particular, $X$ has only Du Bois singularities. 
\end{cor}

By considering the definition and basic properties 
of quasi-log canonical 
pairs, Theorem \ref{f-thm1.1} can be seen as an ultimate 
generalization of \cite[Theorem 1.5]{kollar-kovacs} 
(see also Corollary \ref{f-cor4.2} below). 

\begin{cor}[{\cite[Theorem 1.5]{kollar-kovacs}}]\label{f-cor1.3}
Let $g:X\to Y$ be a proper surjective morphism 
with connected fibers between normal varieties. 
Assume that there exists an effective $\mathbb R$-divisor $\Delta$ on $X$ 
such that $(X, \Delta)$ is log canonical 
and $K_X+\Delta\sim _{\mathbb R, g}0$. 
Then $Y$ has only Du Bois singularities. 

More generally, let $W\subset X$ be a 
reduced closed subscheme that is 
a union of log canonical centers of $(X, \Delta)$. 
Then $g(W)$ is Du Bois. 
\end{cor}

The proof of Theorem \ref{f-thm1.1} in this paper is different from 
the arguments in \cite{kollar-kovacs} and 
\cite{kollar}. 
This is mainly because we can not apply the minimal model 
program 
to quasi-log canonical pairs. 
We will use some kind of canonical bundle formula for reducible varieties, 
which follows from the theory of variations of mixed Hodge 
structure on cohomology with compact support, for the study of 
normal irreducible quasi-log canonical pairs. 
We want to emphasize that we do not use the minimal model 
program in this paper. 
From the Hodge theoretic viewpoint, 
we think that Koll\'ar and Kov\'acs (\cite{kollar-kovacs}) 
avoided 
the use of variations of mixed Hodge structure 
by taking dlt blow-ups, which need the minimal model program. 

\begin{ack}
The first author was partially 
supported by JSPS KAKENHI Grant 
Numbers JP16H03925, JP16H06337. 
\end{ack}

\begin{conventions}
We will work 
over $\mathbb C$, the complex number field, throughout 
this paper. A \emph{scheme} means a separated scheme of finite type over 
$\mathbb C$. 
A {\em{variety}} means a reduced scheme, that is, 
a reduced separated scheme of finite type over $\mathbb C$. 
We will freely use the standard notation of the minimal 
model program and the theory of quasi-log schemes 
as in \cite{fujino-foundations} (see also \cite{fujino-fund}). 
\end{conventions}

\section{Preliminaries}

In this section, let us recall some basic definitions 
and prove an easy lemma. 
 
\medskip

Let $D=\sum _i d_i D_i$ be an $\mathbb R$-divisor, where 
$D_i$ is a prime divisor and $d_i\in \mathbb R$ for every $i$ 
such that $D_i\ne D_j$ for $i\ne j$. 
We put 
\begin{equation*}
D^{<c} =\sum _{d_i<c}d_iD_i, \quad 
D^{\leq c}=\sum _{d_i\leq c} d_i D_i, \quad 
D^{= 1}=\sum _{d_i= 1} D_i, \quad \text{and} \quad
\lceil D\rceil =\sum _i \lceil d_i \rceil D_i, 
\end{equation*}
where $c$ is any real number and 
$\lceil d_i\rceil$ is the integer defined by $d_i\leq 
\lceil d_i\rceil <d_i+1$. 
Moreover, we put $\lfloor D\rfloor =-\lceil -D\rceil$ and 
$\{D\}=D-\lfloor D\rfloor$. 

\medskip 

Let $\Delta_1$ and $\Delta_2$ be $\mathbb R$-Cartier divisors on 
a scheme $X$. 
Then $\Delta_1\sim _{\mathbb R} \Delta_2$ means that $\Delta_1$ is 
$\mathbb R$-linearly equivalent to $\Delta_2$. 

\medskip 

Let $f: X\to Y$ be a morphism 
between schemes and let $B$ be an $\mathbb R$-Cartier divisor 
on $X$. 
Then $B\sim _{\mathbb R, f}0$ means that 
there exists an $\mathbb R$-Cartier divisor $B'$ on $Y$ such that 
$B\sim _{\mathbb R} f^*B'$. 

\medskip

Let $Z$ be a simple normal crossing 
divisor on a smooth variety $M$ and let $B$ be 
an $\mathbb R$-divisor 
on $M$ such that 
$Z$ and $B$ have no common irreducible components and 
that the support of $Z+B$ is a simple normal crossing divisor on $M$. In this 
situation, $(Z, B|_Z)$ is called a {\em{globally embedded simple 
normal crossing pair}}.

\medskip

Let us quickly look at the definition of 
qlc pairs. 

\begin{defn}[Quasi-log canonical pairs]\label{f-def2.1}
Let $X$ be a scheme and let $\omega$ be an 
$\mathbb R$-Cartier divisor (or an $\mathbb R$-line bundle) on $X$. 
Let $f:Z\to X$ be a proper morphism from a globally embedded 
simple normal 
crossing pair $(Z, \Delta_Z)$. If the natural map 
$\mathcal O_X\to f_*\mathcal O_Z(\lceil -(\Delta_Z^{<1})\rceil)$ is an 
isomorphism, $\Delta_Z=\Delta^{\leq 1}_Z$, 
and $f^*\omega\sim _{\mathbb R} K_Z+\Delta_Z$, 
then $\left(X, \omega, f:(Z, \Delta_Z)\to X\right)$ 
is called a {\em{quasi-log canonical pair}} 
({\em{qlc pair}}, for short). If there is no danger of confusion, 
we simply say that $[X, \omega]$ is a qlc pair. 
\end{defn}

Let $\left(X, \omega, f:(Z, \Delta_Z)\to X\right)$ be a quasi-log canonical 
pair as in 
Definition \ref{f-def2.1}. 
Let $\nu:Z^\nu\to Z$ be the normalization. 
We put $K_{Z^\nu}+\Theta=\nu^*(K_Z+\Delta_Z)$, that is, 
$\Theta$ is the sum of the inverse images of $\Delta_Z$ and 
the singular locus of $Z$. 
Then $(Z^\nu, \Theta)$ is sub log canonical in the usual sense. 
Let $W$ be a log canonical center of $(Z^\nu, \Theta)$ or 
an irreducible component of $Z^\nu$. 
Then $f\circ \nu(W)$ is called a {\em{qlc stratum}} of 
$\left(X, \omega, f:(Z, \Delta_Z)\to X\right)$. 
If there is no danger of confusion, we simply call it a qlc stratum of 
$[X, \omega]$. 
If $C$ is a qlc stratum of $[X, \omega]$ but is not 
an irreducible component of $X$, then $C$ is called 
a {\em{qlc center}} of $[X, \omega]$. 
The union of all qlc centers of $[X, \omega]$ is denoted by 
$\Nqklt (X, \omega)$. It is important that 
$[\Nqklt (X, \omega), \omega|_{\Nqklt(X, \omega)}]$ is a 
quasi-log canonical pair by adjunction 
(see, for example, \cite[Theorem 6.3.5]{fujino-foundations}). 

\medskip

Let us prepare an easy lemma for Corollary \ref{f-cor1.3}. 

\begin{lem}\label{f-lem2.2}
Let $[X, \omega]$ be a quasi-log canonical 
pair and let $g:X\to Y$ be a proper morphism 
between varieties with $g_*\mathcal O_X\simeq \mathcal O_Y$. 
Assume that $\omega\sim _{\mathbb R}
g^*\omega'$ holds for some $\mathbb R$-Cartier 
divisor {\em{(}}or $\mathbb R$-line bundle{\em{)}} 
$\omega'$ on $Y$. 
Then $[Y, \omega']$ is a quasi-log canonical 
pair such that $W'$ is a qlc stratum of $[Y, \omega']$ if and only 
if $W'=g(W)$ for some qlc stratum $W$ of $[X, \omega]$. 
\end{lem}

\begin{proof}
By definition, we can take a proper morphism $f:Z\to X$ from a globally 
embedded simple normal crossing pair $(Z, \Delta_Z)$ such that 
$f^*\omega\sim _{\mathbb R} K_Z+\Delta_Z$, 
$\Delta_Z=\Delta^{\leq 1}_Z$, and  
the natural map $\mathcal O_X\to f_*\mathcal O_Z(\lceil -(\Delta^{<1}_Z)
\rceil)$ is an isomorphism. Let us consider $h:=g\circ f: Z\to Y$. 
Then we can easily see that $\left(Y, \omega', h:(Z, \Delta_Z)\to Y\right)$ is 
a quasi-log canonical pair with the desired properties. 
\end{proof}

For the details of the theory of quasi-log 
schemes, see \cite[Chapter 6]{fujino-foundations}. 

\medskip 

Let us recall the definition of semi-log canonical pairs 
for the reader's convenience. 

\begin{defn}[Semi-log canonical pairs]\label{f-def2.3} 
Let $X$ be an equidimensional 
scheme which satisfies Serre's $S_2$ condition and 
is normal crossing in codimension one. 
Let $\Delta$ be an effective $\mathbb R$-divisor 
on $X$ such that no irreducible component of $\Supp \Delta$ is 
contained in the singular locus of $X$ and 
that $K_X+\Delta$ is $\mathbb R$-Cartier. 
We say that $(X, \Delta)$ is a {\em{semi-log canonical}} 
pair if 
$(X^\nu, \Delta_{X^\nu})$ is log canonical 
in the usual sense, where $\nu:X^\nu\to X$ is the normalization of 
$X$ and $K_{X^\nu}+\Delta_{X^\nu}=\nu^*(K_X+\Delta)$, 
that is, $\Delta_{X^\nu}$ is the sum of the inverse images of $\Delta$ 
and the conductor of $X$. 
An {\em{slc center}} of $(X, \Delta)$ is the $\nu$-image 
of an lc center of $(X^\nu, \Delta_{X^\nu})$. 
An {\em{slc stratum}} of $(X, \Delta)$ means 
either an slc center of $(X, \Delta)$ or an irreducible component of 
$X$. 
\end{defn}

For the details of semi-log canonical pairs, see \cite{fujino-slc} 
and \cite{kollar}. 

\section{Du Bois criteria}

A {\em{reduced pair}} is a pair 
$(X, \Sigma)$ where $X$ is a reduced scheme and $\Sigma$ is a reduced closed 
subscheme of $X$. 
Then we can define the {\em{Deligne--Du Bois complex}} 
of $(X, \Sigma)$, which is denoted by $\underline{\Omega}^\bullet_{X, \Sigma}$ 
(see \cite[5.3.1]{fujino-foundations}, 
\cite[Definition 6.4]{kollar}, 
\cite[Definition 3.9]{kovacs1}, \cite[Section 3]{steenbrink}, 
and so on). 
We note that $\underline{\Omega}^\bullet_{X, \Sigma}$ 
is a filtered complex of $\mathcal O_X$-modules in a suitable derived category. 
We put 
$$
\underline{\Omega}^p_{X, \Sigma}:=\Gr^p_{\mathrm{filt}}\!
\underline{\Omega}^\bullet_{X, \Sigma}[p]. 
$$
By taking $\Sigma=\emptyset$, we obtain the {\em{Deligne--Du Bois 
complex}} of $X$: 
$$
\underline{\Omega}^\bullet_X:=\underline{\Omega}^\bullet_{X, \emptyset} 
$$ 
and similarly 
$$
\underline{\Omega}^p_X:=\Gr^p_{\mathrm{filt}}\!
\underline{\Omega}^\bullet_X[p]. 
$$ 
By definition and construction, there exists a natural map 
$$
\mathcal I_{\Sigma}\to \underline{\Omega}^0_{X, \Sigma}, 
$$ 
where $\mathcal I_\Sigma$ is the defining ideal sheaf of 
$\Sigma$ on $X$. 

\begin{defn}[Du Bois pairs and Du Bois singularities]\label{f-def3.1}
A reduced pair $(X, \Sigma)$ is called a 
{\em{Du Bois pair}} if the natural map 
$\mathcal I_\Sigma
\to \underline{\Omega}^0_{X, \Sigma}$ is a quasi-isomorphism. 
When the natural map $\mathcal O_X\to \underline{\Omega}^0_X$ is a 
quasi-isomorphism, we say that $X$ has only {\em{Du Bois singularities}} 
or simply say that $X$ is {\em{Du Bois}}. 
\end{defn}

For the details, see, for example, \cite[Sections 2, 3, 4, and 5]{kovacs1}, 
\cite[Section 6.1]{kollar-kovacs}, and 
\cite[Section 5.3]{fujino-foundations}. 

\medskip 

The following lemma is a very minor modification of 
\cite[Theorem 3.3]{kovacs2} (see also \cite[Theorem 6.27]{kollar}). 
 
\begin{lem}\label{f-lem3.2}
Let $f:Y\to X$ be a proper morphism between varieties. 
Let $V\subset X$ be a closed reduced subscheme with 
ideal sheaf $\mathcal I_V$ and $W\subset Y$ with ideal sheaf 
$\mathcal I_W$. 
Assume that $f(W)\subset V$ and that 
the natural map 
$$
\rho: \mathcal I_V\to Rf_*\mathcal I_W
$$ 
admits a left inverse $\rho'$, that is, $\rho'\circ \rho$ is a quasi-isomorphism. 
In this situation, 
if $(Y, W)$ is a Du Bois pair, then $(X, V)$ is also a Du Bois pair. 
In particular, if $(Y, W)$ is a Du Bois pair, then 
$X$ is Du Bois if and only if $V$ is Du Bois. 
\end{lem} 

We give a proof of Lemma \ref{f-lem3.2} for the reader's convenience 
although it is the same as that of \cite[Theorem 3.3]{kovacs2}. 

\begin{proof}[Proof of Lemma \ref{f-lem3.2}]
By functoriality, we have the 
following commutative diagram 
$$
\xymatrix{
\mathcal I_V \ar[r]^-\rho\ar[d]_-\alpha& Rf_*\mathcal I_W\ar[d]^-\gamma\\ 
\underline{\Omega}^0_{X,V} \ar[r]_-\beta& Rf_*\underline{\Omega}^0_{Y, W}. 
}
$$
Since $(Y, W)$ is assumed to be a Du Bois pair, 
we see that $\gamma$ is a quasi-isomorphism. 
Therefore, $\rho'\circ \gamma^{-1}\circ \beta$ is a left inverse to $\alpha$. 
Then $(X, V)$ is a Du Bois pair by \cite[Theorem 3.2]{kovacs2} (see 
also \cite[Theorem 5.5]{kovacs1} and \cite[Corollary 6.24]{kollar}). 
When $(X, V)$ is a Du Bois pair, it is easy to see that 
$X$ is Du Bois if and only if $V$ is Du Bois by definition 
(see \cite[Proposition 6.15]{kollar} and \cite[Proposition 5.1]{kovacs1}). 
\end{proof}

We prepare one more easy lemma. 

\begin{lem}\label{f-lem3.3}
Let $f:Y\to X$ be a proper birational morphism 
from a smooth irreducible variety $Y$ onto a normal irreducible 
variety $X$. 
Let $B_Y$ be a subboundary $\mathbb R$-divisor 
on $Y$, that is, $B_Y=B^{\leq 1}_Y$, 
such that $\Supp B_Y$ is a simple normal crossing divisor 
on $Y$ and let $M_Y$ be an $f$-nef $\mathbb R$-divisor 
on $Y$. Assume that $K_Y+B_Y+M_Y\sim _{\mathbb R, f}0$, 
$B^{<0}_Y$ is $f$-exceptional, and $f(B^{=1}_Y)$ has only 
Du Bois singularities. 
Then $X$ has only Du Bois singularities. 
\end{lem}
\begin{proof} 
Since $B^{=1}_Y$ is a simple normal crossing divisor on a 
smooth variety $Y$, $B^{=1}_Y$ is Du Bois (see, for example, \cite[Proposition 
5.3.10]{fujino-foundations}). 
In particular, $(Y, B^{=1}_Y)$ is a Du Bois pair. 
We put $Z=f(B^{=1}_Y)$. 
We note that 
$$
-\lfloor B_Y\rfloor -(K_Y+\{B_Y\})\sim _{\mathbb R, f} M_Y
$$ 
is nef and big over $X$. Therefore, $R^if_*\mathcal O_Y(-\lfloor 
B_Y\rfloor)=0$ holds for every $i>0$ by the relative 
Kawamata--Viehweg vanishing theorem. 
Then we have 
$$
\mathcal I_Z= f_*\mathcal O_Y(-B^{=1}_Y)\to Rf_*\mathcal O_Y(-B^{=1}_Y) 
\to Rf_*\mathcal O_Y(-\lfloor B_Y\rfloor)\simeq f_*\mathcal O_Y
(-\lfloor B_Y\rfloor)= \mathcal I_Z, 
$$ 
where $\mathcal I_Z$ is the defining ideal sheaf of $Z$ on $X$. 
This means that the natural map 
$\rho:\mathcal I_Z\to Rf_*\mathcal O_Y(-B^{=1}_Y)$ has a left 
inverse. By Lemma \ref{f-lem3.2}, $(X, Z)$ is a Du Bois pair. 
By assumption, $Z=f(B^{=1}_Y)$ is Du Bois. 
Therefore, we obtain that $X$ has only Du Bois singularities 
(see \cite[Proposition 6.15]{kollar} and \cite[Proposition 5.1]{kovacs1}). 
\end{proof}

\section{Proof of Theorem \ref{f-thm1.1}}

In this section, we prove Theorem \ref{f-thm1.1} 
and Corollaries \ref{f-cor1.2} and \ref{f-cor1.3}. 

\medskip 

First let us prove Theorem 1.1. We note 
that the reduction step to normal irreducible varieties 
in the proof of Theorem 1.1 below is nothing 
but \cite[Proposition 1.4]{fujino-haidong}. 
We include it for the benefit of the reader. 

\begin{proof}[Proof of Theorem \ref{f-thm1.1}] 
Let $[X, \omega]$ be a quasi-log canonical pair. 
We prove Theorem \ref{f-thm1.1} by induction on $\dim X$. 
If $\dim X=0$, then the statement is obvious. 
Let $X_1$ be an irreducible component of $X$ and 
let $X_2$ be the union of the irreducible components 
of $X$ other than $X_1$. 
Then $[X_1, \omega|_{X_1}]$, $[X_2, 
\omega|_{X_2}]$, and $[X_1\cap X_2, 
\omega|_{X_1\cap X_2}]$ are quasi-log canonical pairs 
by adjunction 
(see, for example, \cite[Theorem 6.3.5]{fujino-foundations}). 
In particular, $X_1$, $X_2$, and $X_1\cap X_2$ are 
seminormal (see \cite[Remark 6.2.11]{fujino-foundations}). 
Then we have the following short exact sequence 
\begin{equation}\label{f-eq4.1}
0\to \mathcal O_X\to \mathcal O_{X_1}\oplus \mathcal O_{X_2}\to 
\mathcal O_{X_1\cap X_2}\to 0
\end{equation}
(see, for example, \cite[Lemma 10.21]{kollar}). 
We note that $X_1\cap X_2$ is Du Bois since 
$[X_1\cap X_2, 
\omega|_{X_1\cap X_2}]$ is a quasi-log canonical pair with 
$\dim X_1\cap X_2<\dim X$. 
By \eqref{f-eq4.1} and 
\cite[Lemma 5.3.9]{fujino-foundations}, 
it is sufficient to prove Theorem \ref{f-thm1.1} 
under the extra assumption that 
$X$ is irreducible by induction on 
the number of the irreducible components of $X$. 
Therefore, from now on, we assume that 
$X$ is irreducible. 
Let $\nu:Z\to X$ be the normalization. 
Then, by \cite[Theorem 1.1]{fujino-haidong}, 
$[Z, \nu^*\omega]$ naturally 
becomes a quasi-log canonical pair with 
\begin{equation}\label{f-eq4.2}
R\nu_*\mathcal I_{\Nqklt (Z, \nu^*\omega)}=\mathcal I_{\Nqklt (X, \omega)}. 
\end{equation} 
By induction on dimension, $\Nqklt (Z, \nu^*\omega)$ and 
$\Nqklt (X, \omega)$ are Du Bois. 
This is because 
$[\Nqklt (Z, \nu^*\omega), \nu^*\omega|_{\Nqklt (Z, \nu^*\omega)}]$ 
and $[\Nqklt (X, \omega), \omega|_{\Nqklt (X, \omega)}]$ are quasi-log 
canonical pairs by adjunction (see, for example, 
\cite[Theorem 6.3.5]{fujino-foundations}). 
If $Z$ is Du Bois, then $(Z, \Nqklt(Z, \nu^*\omega))$ is a 
Du Bois pair. In this case, 
we can easily see that $X$ is Du Bois by Lemma \ref{f-lem3.2} 
and \eqref{f-eq4.2} since 
$\Nqklt(X, \omega)$ is Du Bois. 
Therefore, it is sufficient to prove that $Z$ is Du Bois. 
This means that we may assume that $X$ is a normal 
irreducible variety for the proof of Theorem \ref{f-thm1.1}. 
By Theorem \ref{f-thm4.1} below and Lemma \ref{f-lem3.3}, 
we obtain that $X$ has only Du Bois singularities. 
We note that $\Nqklt (X, \omega)$ is Du Bois since 
$[\Nqklt (X, \omega), \omega|_{\Nqklt(X, \omega)}]$ is a 
quasi-log canonical pair with $\dim \Nqklt(X, \omega)
<\dim X$ by adjunction 
(see, for example, \cite[Theorem 6.3.5]{fujino-foundations}). 
Anyway, $X$ is always Du Bois when $[X, \omega]$ is a 
quasi-log canonical pair. 
\end{proof}

The following theorem is a special case of \cite[Theorem 1.5]{fujino-slc-trivial}. 

\begin{thm}[{see \cite[Theorem 1.5]{fujino-slc-trivial}}]\label{f-thm4.1}
Let $[X, \omega]$ be a quasi-log canonical pair such that 
$X$ is a normal irreducible variety. 
Then there exists a projective 
birational morphism 
$p:X'\to X$ from a smooth quasi-projective 
variety $X'$ such that 
$$
K_{X'}+B_{X'}+M_{X'}=p^*\omega, 
$$ 
where $B_{X'}$ is a subboundary $\mathbb R$-divisor, 
that is, $B_{X'}=B^{\leq 1}_{X'}$,  
such that $\Supp B_{X'}$ is a simple normal crossing 
divisor and that $B^{<0}_{X'}$ is 
$p$-exceptional, and $M_{X'}$ is an $\mathbb R$-divisor which is nef 
over $X$. Furthermore, we can make $B_{X'}$ satisfy 
$p(B^{=1}_{X'})=\Nqklt(X, \omega)$. 
\end{thm}

The proof of Theorem \ref{f-thm4.1} uses 
the notion of basic slc-trivial fibrations, which 
is some kind of canonical bundle formula for reducible varieties 
(see \cite{fujino-slc-trivial}). 
Note that \cite{fujino-slc-trivial} depends on some deep 
results on the theory of variations of mixed Hodge structure 
on cohomology with compact support (see 
\cite{fujino-fujisawa}). 

\medskip 

Finally we prove Corollaries \ref{f-cor1.2} and 
\ref{f-cor1.3}. 

\begin{proof}[Proof of Corollary \ref{f-cor1.2}]
Without loss of generality, by shrinking $X$ suitably, 
we may assume that $X$ is quasi-projective 
since the problem is Zariski local. 
Then, by \cite[Theorem 1.2]{fujino-slc}, $[X, K_X+\Delta]$ has a natural 
qlc structure, which is compatible with the original semi-log canonical 
structure. 
For the details, see \cite{fujino-slc}. 
Therefore, any union of slc strata of $(X, \Delta)$ is a 
quasi-log canonical pair. 
Thus, by Theorem \ref{f-thm1.1}, it is Du Bois. 
\end{proof}

\begin{proof}[Proof of Corollary \ref{f-cor1.3}]
Let $f:Z\to X$ be a resolution of singularities such that 
$f^*(K_X+\Delta)=K_Z+\Delta_Z$, 
$\Delta_Z=\Delta^{\leq 1}_Z$, and 
$\Supp \Delta_Z$ is a simple normal crossing divisor 
on $Z$. 
Since $Z$ is irreducible, we can see that $(Z, \Delta_Z)$ is a globally 
embedded simple normal crossing pair. 
Since $\lceil -(\Delta^{<1}_Z)\rceil$ is effective and $f$-exceptional, 
$f_*\mathcal O_Z(\lceil -(\Delta^{<1}_Z)\rceil)\simeq \mathcal O_X$. 
Therefore, $$\left(X, \omega, f:(Z, \Delta_Z)\to X\right),$$ where 
$\omega:=K_X+\Delta$, is a quasi-log canonical 
pair such that $C$ is a log canonical center of $(X, \Delta)$ if and 
only if $C$ is a qlc center of $\left(X, \omega, f: (Z, \Delta_Z)\to X\right)$. 
We take an $\mathbb R$-Cartier divisor (or $\mathbb R$-line bundle) 
$\omega'$ on $Y$ such that 
$\omega\sim _{\mathbb R} g^*\omega'$. 
Then, by Lemma \ref{f-lem2.2}, 
$$\left(Y, \omega', g\circ f: (Z, \Delta_Z)\to Y\right)$$ is a quasi-log 
canonical pair. 
By Theorem \ref{f-thm1.1}, $Y$ is Du Bois. 
We note that $g(W)$ is a union of qlc strata of 
$\left(Y, \omega', g\circ f: (Z, \Delta_Z)\to Y\right)$. 
Therefore, by adjunction (see, 
for example, \cite[Theorem 6.3.5]{fujino-foundations}), 
$[g(W), \omega'|_{g(W)}]$ is also a quasi-log canonical 
pair. Then, by Theorem \ref{f-thm1.1} again, 
$g(W)$ has only Du Bois singularities. 
\end{proof}

By combining \cite[Theorem 1.2]{fujino-slc} with Lemma \ref{f-lem2.2}, 
we can generalize Corollary \ref{f-cor1.3} as follows. 

\begin{cor}\label{f-cor4.2}
Let $g:X\to Y$ be a projective surjective 
morphism between varieties with $g_*\mathcal O_X\simeq 
\mathcal O_Y$. 
Assume that there exists an effective $\mathbb R$-divisor 
$\Delta$ on $X$ such that $(X, \Delta)$ is semi-log canonical 
and $K_X+\Delta\sim _{\mathbb R, g}0$. 
Then $Y$ has only Du Bois singularities. 

More generally, let $W\subset X$ be a reduced closed subscheme 
that is a union of slc strata of $(X, \Delta)$. 
Then $g(W)$ is Du Bois. 
\end{cor}

%%%%%%%%%%%%%%%

\end{document}